\documentclass[10pt]{article}

\usepackage{amsmath,amsbsy,amssymb,amsthm}
\usepackage{a4wide}
\usepackage{graphicx,psfrag}

\theoremstyle{plain}
\newtheorem{theorem}{Theorem}[section]
\newtheorem{lemma}[theorem]{Lemma}
\newtheorem{example}[theorem]{Example}

\theoremstyle{remark}
\newtheorem{remark}{Remark}

\theoremstyle{definition}
\newtheorem{definition}{Definition}

\def\s{\sigma}
\def\xs{_\s[x]_\gamma^{\alpha, \beta}}

\def\a{{\alpha(\cdot,\cdot)}}
\def\an{{\alpha_n(\cdot,\cdot)}}
\def\ao{{\alpha_1(\cdot,\cdot)}}

\def\ai{{\alpha_i(\cdot,\cdot)}}
\def\b{{\beta(\cdot,\cdot)}}
\def\bn{{\beta_n(\cdot,\cdot)}}
\def\bo{{\beta_1(\cdot,\cdot)}}

\def\bi{{\beta_i(\cdot,\cdot)}}
\def\t{\tau}
\def\e{\epsilon}

\def\LIan{{_aI_t^{\an}}}
\def\RIan{{_tI_b^{\an}}}
\def\LDa{{_aD_t^{\a}}}
\def\LDan{{_aD_t^{\an}}}
\def\LDb{{_aD_t^{\b}}}
\def\RDa{{_tD_b^{\a}}}
\def\RDan{{_tD_b^{\an}}}

\def\RDbn{{_tD_b^{\bn}}}
\def\LC{{^C_aD_t^{\a}}}
\def\LCan{{^C_aD_t^{\an}}}
\def\LCao{{^C_aD_t^{\ao}}}

\def\RCa{{^C_tD_b^{\a}}}
\def\RCan{{^C_tD_b^{\an}}}

\def\RCb{{^C_tD_b^{\b}}}
\def\RCbo{{^C_tD_b^{\bo}}}

\def\RCbn{{^C_tD_b^{\bn}}}
\def\DC{{^CD_\gamma^{\a,\b}}}
\def\DCi{{^CD_{\gamma^i}^{\ai,\bi}}}
\def\DS{\displaystyle}

\begin{document}

\title{Combined Fractional Variational Problems 
of Variable Order and Some Computational Aspects\thanks{This 
is a preprint of a paper whose final and definite form is with 
\emph{Journal of Computational and Applied Mathematics}, ISSN: 0377-0427. 
Submitted 15-Feb-2017; Revised 17-Apr-2017;
Accepted 21-Apr-2017.}}

\author{Dina Tavares$^{\rm a,b}$\\
{\tt dtavares@ipleiria.pt}
\and Ricardo Almeida$^{\rm b}$\\
{\tt ricardo.almeida@ua.pt} 
\and Delfim F. M. Torres$^{\rm b}$\\
{\tt delfim@ua.pt}}

\date{$^{a}${\em{ESECS, Polytechnic Institute of Leiria, 2411--901 Leiria, Portugal}}\\[0.3cm]
$^{b}$\text{\em{Center for Research and Development in Mathematics and Applications (CIDMA)}},
\em{Department of Mathematics, University of Aveiro, 3810--193 Aveiro, Portugal}}

\maketitle


\begin{abstract}
We study two generalizations of fractional variational problems 
by considering higher-order derivatives and a state time delay. 
We prove a higher-order integration by parts formula involving 
a Caputo fractional derivative of variable order and we establish 
several necessary optimality conditions for functionals containing 
a combined Caputo derivative of variable fractional order. Because 
the endpoint is considered to be free, we also deduce associated 
transversality conditions. In the end, we consider functionals 
with a time delay and deduce corresponding optimality conditions. 
Some examples are given to illustrate the new results. Computational 
aspects are discussed using the open source software package \textsf{Chebfun}.

\bigskip


\smallskip

\end{abstract}


\small

\section{Introduction}

Fractional Calculus (FC) is an extension of the integer-order calculus
that considers derivatives of any real or complex order \cite{Kilbas,Samko_1}. 
FC was born in 1695 with a letter that L'H\^opital wrote to Leibniz, 
where the derivative of order $1/2$ is suggested \cite{Oldham}. Since then, 
many mathematicians, like Laplace, Riemann, Liouville, Abel, among others, 
contributed to the development of this subject. One of the first applications 
of fractional calculus was due to Abel in his solution to the tautochrone problem \cite{Abel}.
Different forms of fractional operators have been introduced along time, 
like the Riemann--Liouville, the Riesz or the Caputo fractional derivatives. 
For new kinds of fractional derivatives with nonsingular kernels, see \cite{Abd:Bal:2017}.
In this paper, we are interested on the combined Caputo derivative 
$^{C}D_{\gamma}^{\alpha,\beta}$ \cite{Malin_Tor}, 
which is a convex combination of the left Caputo fractional derivative 
of order $\alpha$ and the right Caputo fractional derivative of order $\beta$.

In recent times, FC had an increasing of importance due to its applications 
in various fields, not only in mathematics, but also in physics, engineering, 
chemistry, biology, finance and other areas of science 
\cite{Li,Od_GFCV2013,Pinto,Sie,Sun}. In some of these applications, 
if we compare with the usual integer-order calculus, FC is better to describe 
the hereditary and memory properties of materials and processes. More interesting
possibilities arise when one considers the order $\alpha$ of the fractional integrals 
and derivatives not constant during the process but depending on time. One such 
fractional calculus of variable order was introduced in 1993 by Samko and Ross \cite{SamkoRoss}. 
Afterwards, several mathematicians obtained important results about variable order fractional calculus, 
see, for instance, \cite{Atana,Od_FVC2013,Samko_2}. Here, we consider the combined Caputo fractional 
derivative of a function $x$ with variable order, defined by
$^{C}D_{\gamma}^{\a,\b}x(t)=\gamma_1 \, \LC x(t)+\gamma_2 \, {^C_tD_b^{\b}} x(t)$
with $\gamma =\left(\gamma_1,\gamma_2 \right)\in [0,1]^2$.
Some numerical approximate formulas for such fractional calculus have been proposed, 
see, for example, \cite{Li_Chen,Tavares_2}.

The Fractional Calculus of Variations (FCV) deals with the optimization of functionals
that depend on some fractional operator \cite{book:APT,book:MOT,MR2984893}. 
The fundamental problem is to find functions that extremize (minimize or maximize) 
such a functional. Although FCV was born only twenty years ago, with the 1996--1997 works 
of Riewe in mechanics \cite{Riewe}, in ours days, this is a strong field of mathematics: 
see, e.g., \cite{Almeida,Askari,Malin_Tor,MR2984893,Tavares_3}.
The main goal of our work is to generalize the results obtained in \cite{Tavares_1} 
by considering higher-order and time delay variational problems with a Lagrangian 
depending on a combined Caputo derivative of variable fractional order, 
subject to boundary conditions at the initial time $t=a$.

The outline of the paper is as follows. In Section~\ref{sec:FC}, 
we review the necessary notions on fractional calculus and prove 
an integration by parts formula, involving a higher-order Caputo 
fractional derivative of variable order (see Theorem~\ref{thm:FIP_HO}). 
In Section~\ref{sec:highorder}, we obtain higher-order Euler--Lagrange equations 
and transversality conditions for the generalized variational problems 
with a Lagrangian depending on a combined Caputo derivative of variable fractional order 
(Theorem~\ref{HO_teo1}). Then, in Section~\ref{sec:delay}, 
we deduce necessary optimality conditions when the Lagrangian depends on a time delay. 
Some illustrate examples are presented in Section~\ref{sec:examples}.
We end with Section~\ref{sec:conc} of conclusion and an appendix
with our \textsf{MATLAB} code.


\section{Fractional calculus of variable order}
\label{sec:FC}

In this section, we review some concepts about the operators 
that are used in the sequel. For more information about 
the theory of fractional calculus, see, for example, \cite{Kilbas,Samko_1}.

The Gamma function $\Gamma$ is an extension of the factorial
to real numbers, while the Beta function $B$ is defined by
$B(t,u)=\displaystyle \int_0^1 s^{t-1} (1-s)^{u-1}ds$, $t,u > 0$.
For numerical computations, we have used \textsf{MATLAB} \cite{Matlab}
and \textsf{Chebfun} \cite{chebfun:book}. Both functions $\Gamma(t)$ and $B(t,u)$
are available in \textsf{Matlab} through the commands
\texttt{gamma(t)} and \texttt{beta(t,u)}, respectively.
This second function satisfies the property 
$B(t,u)=\displaystyle \frac{\Gamma(t)\Gamma(u)}{\Gamma(t+u)}$.

Motivated by Definitions 2.3 and 2.5 of \cite{book:MOT}, we present 
here the concepts of higher-order fractional derivatives 
of Riemann--Liouville and Caputo. Let $ n \in \mathbb{N}$ and 
$x: [a,b] \rightarrow \mathbb{R}$ be a function of class $C^n$. 
The fractional order is a continuous function of two variables, 
$\alpha_n:[a,b]^2\to (n-1,n)$.

\begin{definition}[Higher-order Riemann--Liouville fractional derivatives]
\label{HORLFD}
The left and right Riemann--Liouville fractional derivatives of order $\an$ are defined by
$$
\LDan x(t)=\frac{d^n}{dt^n}\int_a^t \frac{1}{\Gamma(n-\alpha_n(t,\t))}(t-\t)^{n-1-\alpha_n(t,\t)}x(\t)d\t
$$
and
$$
\RDan x(t)=(-1)^{n}\frac{d^n}{dt^n}\int_t^b\frac{1}{\Gamma(n-\alpha_n(\t,t))}(\t-t)^{n-1-\alpha_n(\t,t)} x(\t)d\t,
$$
respectively.
\end{definition}

In our work, we use both Riemann--Liouville and Caputo definitions.
The emphasis is, however, in Caputo fractional derivatives.

\begin{definition}[Higher-order Caputo fractional derivatives]
\label{HOCFD}
The left and right Caputo fractional derivatives of order $\an$ are defined by
\begin{equation}
\label{eq:left:Cap:der}
\LCan x(t)=\int_a^t\frac{1}{\Gamma(n-\alpha_n(t,\t))}(t-\t)^{n-1-\alpha_n(t,\t)}x^{(n)}(\t)d\t
\end{equation}
and
\begin{equation}
\label{eq:right:Cap:der}
\RCan x(t)=(-1)^{n}\int_t^b\frac{1}{\Gamma(n-\alpha_n(\t,t))}(\t-t)^{n-1-\alpha_n(\t,t)}x^{(n)}(\t)d\t,
\end{equation}
respectively.
\end{definition}

\begin{remark}
Definitions~\ref{HORLFD} and \ref{HOCFD}, for the particular case
of order between 0 and 1, can be found in \cite{book:MOT}.
They seem to be new for the higher-order case. 
\end{remark}

\textsf{Chebfun} is a \textsf{MATLAB} software system that overloads 
\textsf{MATLAB}'s discrete operations for matrices to analogous 
continuous operations for functions and operators \cite{chebfun:book}. 
Using Definition~\ref{HOCFD}, we implemented in \textsf{Chebfun}
two functions \texttt{leftCaputo(x,alpha,a,n)} and \texttt{rightCaputo(x,alpha,b,n)}
that approximate, respectively, the higher-order Caputo fractional derivatives
$\LCan x(t)$ and $\RCan x(t)$: see Appendix~\ref{Chebfun:Caputo:der}.
Follows two illustrative examples.

\begin{example}
\label{ex:cheb01}
Let $\alpha(t,\tau) = \frac{t^2}{2}$ and $x(t) = t^4$ with $t \in [0,1]$.
In this case, $a = 0$, $b = 1$ and $n = 1$. We have 
${^C_aD_{0.6}^{\alpha(\cdot,\cdot)}} x(0.6) \approx 0.1857$ 
and ${^C_{0.6}D_b^{\alpha(\cdot,\cdot)}} x(0.6) \approx -1.0385$, obtained
in \textsf{Matlab} with our \textsf{Chebfun} functions as follows:
{\small \begin{verbatim}
  a = 0; b = 1; n = 1;
  alpha = @(t,tau) t.^2/2;
  x = chebfun(@(t) t.^4, [a b]);
  LC = leftCaputo(x,alpha,a,n);
  RC = rightCaputo(x,alpha,b,n);
  LC(0.6)
  ans = 0.1857  
  RC(0.6)
  ans = -1.0385  
\end{verbatim}}
\noindent See Figure~\ref{fig:ex1} for a plot
with other values of ${^C_aD_{t}^{\alpha(\cdot,\cdot)}} x(t)$ 
and ${^C_{t}D_b^{\alpha(\cdot,\cdot)}} x(t)$. 
\begin{figure}[htb]
\centering
\includegraphics[scale=0.5]{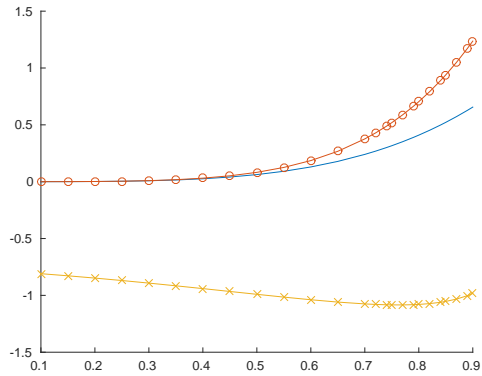}
\caption{Caputo fractional derivatives of Example~\ref{ex:cheb01}:
$x(t) = t^4$ in continuous line, left derivative ${^C_aD_t^{\alpha(\cdot,\cdot)}} x(t)$ 
with ``$\circ -$'' style, and right derivative ${^C_tD_b^{\alpha(\cdot,\cdot)}} x(t)$ 
with ``$\times -$'' style.}
\label{fig:ex1}
\end{figure}
\end{example}

\begin{example}
\label{ex:cheb01b}
In Example~\ref{ex:cheb01}, we have used the polynomial $x(t) = t^4$. 
It is worth mentioning that our \textsf{Chebfun} implementation
works well for functions that are not a polynomial. For example,
let $x(t) = e^t$. In this case, we just need to change
{\small \begin{verbatim}
  x = chebfun(@(t) t.^4, [a b]);
\end{verbatim}}
\noindent in Example~\ref{ex:cheb01} by 
{\small \begin{verbatim}
  x = chebfun(@(t) exp(t), [a b]);
\end{verbatim}}
\noindent to obtain
{\small \begin{verbatim}
  LC(0.6)
  ans = 0.9917  
  RC(0.6)
  ans = -1.1398  
\end{verbatim}}
\noindent See Figure~\ref{fig:ex1b} for a plot
with other values of ${^C_aD_{t}^{\alpha(\cdot,\cdot)}} x(t)$ 
and ${^C_{t}D_b^{\alpha(\cdot,\cdot)}} x(t)$. 
\begin{figure}[htb]
\centering
\includegraphics[scale=0.5]{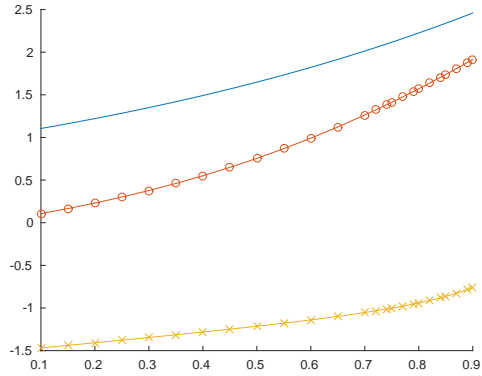}
\caption{Caputo fractional derivatives of Example~\ref{ex:cheb01b}:
$x(t) = e^t$ in continuous line, left derivative ${^C_aD_t^{\alpha(\cdot,\cdot)}} x(t)$ 
with ``$\circ -$'' style, and right derivative ${^C_tD_b^{\alpha(\cdot,\cdot)}} x(t)$ 
with ``$\times -$'' style.}
\label{fig:ex1b}
\end{figure}
\end{example}

Now, we define the generalized fractional integrals for variable order.

\begin{definition}[Riemann--Liouville fractional integrals] 
\label{def:RL:fi}
The left and right Riemann--Liouville fractional integrals of order $\an$ 
are defined respectively by
$$
\LIan x(t)=\int_a^t  \frac{1}{\Gamma(\alpha_n(t,\t))}(t-\t)^{\alpha_n(t,\t)-1}x(\t)d\t
$$
and
$$
\RIan x(t)=\int_t^b\frac{1}{\Gamma(\alpha_n(\t,t))}(\t-t)^{\alpha_n(\t,t)-1}x(\t)d\t.
$$
\end{definition}

Our \textsf{Chebfun} definitions of the
Riemann--Liouville fractional integrals
are given in Appendix~\ref{Chebfun:RL:int}.
Here we illustrate their use.

\begin{example}
\label{ex:cheb02}
Let $\alpha(t,\tau) = \frac{t^2+\tau^2}{4}$ and $x(t) = t^2$ with $t \in [0,1]$.
In this case, $a = 0$, $b = 1$ and $n = 1$. We have 
${_aI_{0.6}^{\alpha(\cdot,\cdot)}} x(0.6) \approx 0.2661$ 
and ${_{0.6}I_b^{\alpha(\cdot,\cdot)}} x(0.6) \approx 0.4619$, obtained
in \textsf{Matlab} with our \textsf{Chebfun} functions as follows:
{\small \begin{verbatim}
  a = 0; b = 1; n = 1;
  alpha = @(t,tau) (t.^2+tau.^2)/4; 
  x = chebfun(@(t) t.^2, [0,1]);
  LFI = leftFI(x,alpha,a); 
  RFI = rightFI(x,alpha,b);
  LFI(0.6)
  ans = 0.2661
  RFI(0.6)
  ans = 0.4619  
\end{verbatim}}
\noindent Other values for ${_aI_{t}^{\alpha(\cdot,\cdot)}} x(t)$ 
and ${_{t}I_b^{\alpha(\cdot,\cdot)}} x(t)$ 
are plotted in Figure~\ref{fig:ex2}.
\begin{figure}[htb]
\centering
\includegraphics[scale=0.5]{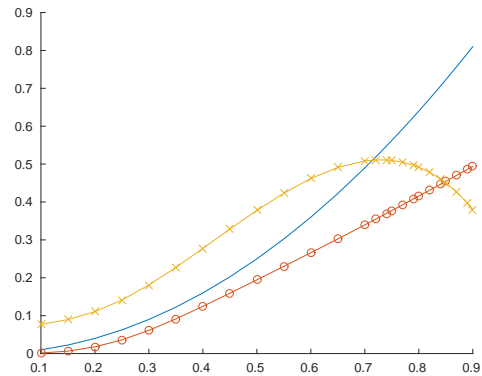}
\caption{Riemann--Liouville fractional integrals of Example~\ref{ex:cheb02}:
$x(t) = t^2$ in continuous line, left integral ${_aI_t^{\alpha(\cdot,\cdot)}} x(t)$ 
with ``$\circ -$'' style, and right integral ${_tI_b^{\alpha(\cdot,\cdot)}} x(t)$ 
with ``$\times -$'' style.}
\label{fig:ex2}
\end{figure}
\end{example}

\begin{remark}
From Definition~\ref{def:RL:fi}, it follows that
$$
\LDan x(t)=\frac{d^{n}}{dt^{n}}\, _aI_t^{n-\an} x(t), 
\quad\RDan x(t)=(-1)^{n}\frac{d^{n}}{dt^{n}}\, _tI_b^{n-\an} x(t)
$$
and
$$
\LCan x(t)=\, _aI_t^{n-\an} \frac{d^{n}}{dt^{n}}x(t), 
\quad\RCan x(t)=(-1)^{n}\, _tI_b^{n-\an} \frac{d^{n}}{dt^{n}}x(t).
$$
\end{remark}

Next, we obtain higher-order Caputo fractional derivatives of a power function. 
This allows us to show the effectiveness of our computational approach, that is,
the usefulness of polynomial interpolation in Chebyshev points in fractional calculus
of variable order. In Lemma~\ref{lemma:power}, we assume that the fractional order 
depends only on the first variable: $\alpha_n(t,\t) := \overline{\alpha}_n(t)$, where  
$\overline{\alpha}_n:[a,b]\to (n-1,n)$ is a given function.

\begin{lemma}
\label{lemma:power}
Let $x(t)=(t-a)^\gamma$ with $\gamma>n-1$. Then,
$$ 
{^C_aD_t^{\overline{\alpha}_n(t)}}x(t) 
= \frac{\Gamma(\gamma+1)}{\Gamma(\gamma-\overline{\alpha}_n(t)+1)}
(t-a)^{\gamma-\overline{\alpha}_n(t)}.
$$
\end{lemma}

\begin{proof} 
As $x(t)=(t-a)^\gamma$, if we differentiate it $n$ times, we obtain
$$
x^{(n)}(t)=\frac{\Gamma(\gamma + 1)}{\Gamma(\gamma-n+1)}(t-a)^{\gamma-n}.
$$
Using Definition~\ref{HOCFD} of the left Caputo fractional derivative, we get
$$
\begin{array}{ll}
{^C_aD_t^{\overline{\alpha}_n(t)}}x(t) 
&= \DS \int_a^t \frac{1}{\Gamma(n-\overline{\alpha}_n(t))}
(t-\t)^{n-1-\overline{\alpha}_n(t)}x^{(n)}(\t) d\t\\
&=\DS \int_a^t \frac{\Gamma(\gamma + 1)}{\Gamma(\gamma-n+1)
\Gamma(n-\overline{\alpha}_n(t))}(t-\t)^{n-1-\overline{\alpha}_n(t)}(\t-a)^{\gamma-n} d\t.
\end{array}
$$
Now, we proceed with the change of variables $\t-a=s(t-a)$. 
Using the Beta function $B(\cdot,\cdot)$, we obtain that
$$
\begin{array}{ll}
{^C_aD_t^{\overline{\alpha}_n(t)}}x(t) 
&= \DS \frac{\Gamma(\gamma + 1)}{\Gamma(\gamma-n+1)\Gamma(n-\overline{\alpha}_n(t))} 
\int_0^1 (1-s)^{n-1-\overline{\alpha}_n(t)} s^{\gamma-n} (t-a)^{\gamma-\overline{\alpha}_n(t)} ds\\
&= \DS \frac{\Gamma(\gamma + 1)(t-a)^{\gamma-\overline{\alpha}_n(t)}}{\Gamma(\gamma-n+1)
\Gamma(n-\overline{\alpha}_n(t))} B \left(\gamma-n+1, n-\overline{\alpha}_n(t)\right).\\
&= \DS \frac{\Gamma(\gamma + 1)}{\Gamma(\gamma-\overline{\alpha}_n(t)+1)} 
\quad (t-a)^{\gamma-\overline{\alpha}_n(t)}.
\end{array}
$$
The proof is complete.
\end{proof}

\begin{example}
\label{ex:lemma:power}
Let us revisit Example~\ref{ex:cheb01} by choosing
$\alpha(t,\tau) = \frac{t^2}{2}$ and $x(t) = t^4$ with $t \in [0,1]$.
Table~\ref{table01} shows the approximated values obtained
by our \textsf{Chebfun} function \texttt{leftCaputo(x,alpha,a,n)}
and the exact values computed with the formula given by Lemma~\ref{lemma:power}.
Table~\ref{table01} was obtained using the following \texttt{MATLAB} code:
{\small \begin{verbatim}
  format long
  a = 0; b = 1; n = 1;
  alpha = @(t,tau) t.^2/2;
  x = chebfun(@(t) t.^4, [a b]);
  exact = @(t) (gamma(5)./gamma(5-alpha(t))).*t.^(4-alpha(t)); 
  approximation = leftCaputo(x,alpha,a,n); 
  for i = 1:9
    t = 0.1*i;      
    E = exact(t); 
    A = approximation(t);
    error = E - A;
    [t E A error]
  end
\end{verbatim}}
\begin{table}
\begin{center}
\begin{tabular}{|c|c|c|c|} \hline
$\mathbf{t}$ & \textbf{Exact Value} & \textbf{Approximation} & \textbf{Error} \\ \hline
0.1 & 1.019223177296953e-04 &  1.019223177296974e-04 & -2.046431600566390e-18 \\ \hline
0.2 & 0.001702793965464 &  0.001702793965464 & -2.168404344971009e-18 \\ \hline
0.3 & 0.009148530806348 &  0.009148530806348 &  3.469446951953614e-18 \\ \hline
0.4 & 0.031052290994593 &  0.031052290994592 &  9.089951014118469e-16 \\ \hline
0.5 & 0.082132144921157 &  0.082132144921157 &  6.522560269672795e-16 \\ \hline
0.6 & 0.185651036003120 &  0.185651036003112 &  7.938094626069869e-15 \\ \hline
0.7 & 0.376408251363662 &  0.376408251357416 &  6.246059225389899e-12 \\ \hline
0.8 & 0.704111480975332 &  0.704111480816562 &  1.587694420379648e-10 \\ \hline
0.9 & 1.236753486749357 &  1.236753486514274 &  2.350835082154390e-10 \\ \hline
\end{tabular} 
\end{center}
\caption{Exact values obtained by Lemma~\ref{lemma:power}
for functions of Example~\ref{ex:lemma:power} 
versus computational approximations obtained using  
our \textsf{Chebfun} code of Appendix~\ref{Chebfun:code}.}
\label{table01}
\end{table}
\end{example}

For our next result, we assume that the fractional order depends only on the second variable:  
$\alpha_n(\t,t) := \overline{\alpha}_n(t)$, where  $\overline{\alpha}_n : [a,b] \to (n-1,n)$ is
a given function. The proof is similar to that of 
Lemma~\ref{lemma:power}, and so we omit it here.

\begin{lemma}
\label{lemma:power:2}
Let $x(t)=(b-t)^\gamma$ with $\gamma>n-1$. Then,
$$ 
{^C_tD_b^{\overline{\alpha}_n(t)}}x(t) 
= \frac{\Gamma(\gamma+1)}{\Gamma(\gamma-\overline{\alpha}_n(t)+1)}
(b-t)^{\gamma-\overline{\alpha}_n(t)}.
$$
\end{lemma}

Computational experiments similar to those of Example~\ref{ex:lemma:power},
obtained by substituting Lemma~\ref{lemma:power} by Lemma~\ref{lemma:power:2} 
and our \texttt{leftCaputo} routine by the \texttt{rightCaputo} one, 
reinforce the validity of our computational methods.

When dealing with variational problems, one key property is 
integration by parts. In the following theorem, such formulas 
are proved for integrals involving higher-order Caputo 
fractional derivatives of variable order.

\begin{theorem}[Integration by parts]
\label{thm:FIP_HO}
Let $n \in \mathbb{N}$ and $x,y \in C^m\left([a,b], \mathbb{R}\right)$ be two functions. Then,
$$
\int_{a}^{b}y(t) \, \LCan x(t)dt
=\int_a^b x(t) \, {\RDan}y(t)dt+\left[\sum_{k=0}^{n-1} (-1)^{k} x^{(n-1-k)}(t) \, 
\dfrac{d^{k}}{dt^{k}}{_tI_b^{n-\an}}y(t) \right]_{a}^{b}
$$
and
$$
\int_{a}^{b}y(t) \, {\RCan}x(t)dt=\int_a^b x(t) \, {\LDan} y(t)dt
+\left[\sum_{k=0}^{n-1} (-1)^{n+k}x^{(n-1-k)}(t) \, 
\dfrac{d^{k}}{dt^{k}}{_aI_t^{n-\an}}y(t)\right]_{a}^{b}.
$$
\end{theorem}

\begin{proof}
Considering the definition of left Caputo fractional derivative of order $\an$, we obtain
$$
\int_{a}^{b}y(t) \, \LCan x(t)dt
=\int_a^b \int_a^t y(t) \, 
\dfrac{1}{\Gamma(n-\alpha_n(t,\t))} (t-\t)^{n-1-\alpha_n(t,\t)}x^{(n)}(\t)d\t dt.
$$
Using Dirichelet's formula, we rewrite it as
\begin{equation}
\label{eq:star}
\int_a^b \int_t^b y(\t) \, 
\dfrac{(\t-t)^{n-1-\alpha_n(\t,t)}}{\Gamma(n-\alpha_n(\t,t))} x^{(n)}(t)d\t dt
=\int_a^b x^{(n)}(t)  \, {_tI_b^{n-\an}} y(t) dt.
\end{equation}
Using the (usual) integrating by parts formula, we get that \eqref{eq:star} is equal to
$$
-\int_a^b x^{(n-1)}(t) \dfrac{d}{dt}  {_tI_b^{n-\an}} y(t)dt 
+ \left[ x^{(n-1)}(t) {_tI_b^{n-\an}}y(t)\right]_{a}^{b}. 
$$
Integrating by parts again, we obtain
\begin{equation*}
\begin{split}
\int_a^b x^{(n-2)}(t) \, \dfrac{d^2}{dt^2}  {_tI_b^{n-\an}} y(t)dt
+ \left[ x^{(n-1)}(t) \, {_tI_b^{n-\an}}y(t) - x^{(n-2)}(t) 
\dfrac{d}{dt} {_tI_b^{n-\an}}y(t)\right]_{a}^{b}.
\end{split}
\end{equation*}
If we repeat this process $n-2$ times more, we get
\begin{multline*}
\int_a^b x(t) (-1)^n \, \dfrac{d^n}{dt^n}  {_tI_b^{n-\an}} y(t)dt 
+ \left[ \sum_{k=0}^{n-1} (-1)^{k} x^{(n-1-k)}(t) \, 
\dfrac{d^{k}}{dt^{k}} {_tI_b^{n-\an}}y(t)\right]_{a}^{b} \\
=\int_a^b x(t) \RDan y(t) dt + \left[ \sum_{k=0}^{n-1} 
(-1)^{k} x^{(n-1-k)}(t) \, \dfrac{d^{k}}{dt^{k}} {_tI_b^{n-\an}}y(t)\right]_{a}^{b}.
\end{multline*}
The second relation of the theorem for the right Caputo fractional 
derivative of order $\an$ follows directly, from the first, 
by Caputo--Torres duality \cite{MyID:307}.
\end{proof}

\begin{remark} 
If we consider in Theorem~\ref{thm:FIP_HO}
the particular case when $n=1$, then the fractional integration 
by parts formulas take the following well-known forms 
$$
\int_{a}^{b}y(t) \, \LC x(t)dt
=\int_a^b x(t) \, {\RDa}y(t)dt
+\left[x(t) \, {_tI_b^{1-\a}}y(t) \right]_{a}^{b}
$$
and
$$
\int_{a}^{b}y(t) \, {\RCa}x(t)dt
=\int_a^b x(t) \, {\LDa} y(t)dt
-\left[x(t) \, {_aI_t^{1-\a}}y(t)\right]_{a}^{b}
$$
(see, e.g., \cite[Theorem~3.2]{Od_NTFVP2013}).
\end{remark}

\begin{remark}
If $x$ is such that $x^{(i)}(a)=x^{(i)}(b)=0$ for all $i=0, \ldots, n-1$, 
then the higher-order formulas of fractional integration by parts given 
by Theorem~\ref{thm:FIP_HO} can be rewritten as
$$
\int_{a}^{b}y(t) \, \LCan x(t)dt=\int_a^b x(t) \, {\RDan}y(t)dt,
\quad
\int_{a}^{b}y(t) \, {\RCan}x(t)dt=\int_a^b x(t) \, {\LDan} y(t)dt.
$$
\end{remark}

The next step is to consider a linear combination of the previous fractional derivatives.

\begin{definition}[Higher-order combined fractional derivatives]
\label{def1}
Let $\alpha_n,\beta_n: [a,b]^2\rightarrow(n-1,n)$ be the variable fractional order, 
$\gamma^n =\left(\gamma_1^n,\gamma_2^n \right)\in [0,1]^2$ a vector, 
and $x \in C^n\left([a,b]\right)$. The higher-order combined Riemann--Liouville 
fractional derivative of $x$ at $t$ is defined by
$$
D_{\gamma^n}^{\an,\bn}x(t)=\gamma_1^n \, \LDan x(t)+\gamma_2^n \, \RDbn x(t).
$$
Analogously, the higher-order combined Caputo fractional derivative of $x$ at $t$ is defined by
$$
^{C}D_{\gamma^n}^{\an,\bn}x(t)=\gamma_1^n \, \LCan x(t)+\gamma_2^n \, \RCbn x(t).
$$
\end{definition}

See our \textsf{Chebfun} computational code for the higher-order 
combined fractional Caputo derivative in Appendix~\ref{Chebfun:combined}.
Here we illustrate how to use it in \textsf{MATLAB}.

\begin{example}
\label{ex:matlab:comb}
Let $\alpha(t,\tau) = \frac{t^2 + \tau^2}{4}$,
$\beta(t,\tau) = \frac{t + \tau}{3}$ and $x(t) = t$,
$t\in [0,1]$. We have $a = 0$, $b = 1$ and $n = 1$.
For $\gamma = (\gamma_1, \gamma_2) = (0.8, 0.2)$, we have
$^{C}D_{\gamma}^{\alpha(\cdot,\cdot),\beta(\cdot,\cdot)}x(0.4) \approx 0.7144$:
{\small \begin{verbatim}
  a = 0; b = 1; n = 1;
  alpha = @(t,tau) (t.^2 + tau.^2)/.4;
  beta = @(t,tau) (t + tau)/3; 
  x = chebfun(@(t) t, [0 1]);
  gamma1 = 0.8;
  gamma2 = 0.2;
  CC = combinedCaputo(x,alpha,beta,gamma1,gamma2,a,b,n);
  CC(0.4)
  ans = 0.7144
\end{verbatim}}
\end{example}


\section{Variational problems with higher-order derivatives}
\label{sec:highorder}

Given $n\in \mathbb{N}$, let $D$ denote the linear subspace of $C^n([a,b])\times [a,b]$
such that $\DCi x(t)$ exists and is continuous on the interval $[a,b]$ 
for all $i \in \{1, \ldots,n\}$. We endow $D$ with the norm
$$
\|(x,t)\|=\max_{a\leq t \leq  b}|x(t)|+\max_{a\leq t \leq b}
\sum_{i=1}^{n} \left| \DCi  x(t)\right|+|t|.
$$
Consider the following problem of the calculus of variations: 
minimize functional $\mathcal{J}:D\rightarrow \mathbb{R}$, 
\begin{equation}
\label{funct_HO_1}
\mathcal{J}(x,T)
=\int_a^T L\left(t, x(t), \,^CD_{\gamma^1}^{\alpha_1(\cdot,\cdot), 
\beta_1(\cdot,\cdot)}x(t), \ldots, \,^CD_{\gamma^n}^{\alpha_n(\cdot,\cdot), 
\beta_n(\cdot,\cdot)}x(t)  \right) dt + \phi(T,x(T)),
\end{equation}
over all $(x,T)\in D$ subject to boundary conditions
$x(a)=x_a$, $x^{(i)}(a)=x^{i}_a$, $\forall i \in \lbrace1,\ldots,n-1\rbrace$, 
for fixed $x_a, x^{1}_a, \ldots, x^{n-1}_a \in \mathbb{R}$. 
Here the terminal time $T$ and terminal state $x(T)$ are both free. 
For all $i \in \lbrace1, \ldots, n\rbrace$, 
$\alpha_i, \beta_i  \left([a,b]^2\right) \subseteq (i-1,i)$ 
and $\gamma^i=\left(\gamma_1^i, \gamma_2^i\right)$ is a vector. 
The terminal cost function $\phi:[a,b]\times \mathbb{R}\rightarrow \mathbb{R}$ is at least of class $C^1$.
For simplicity of notation, we introduce the operator $[\cdot]_\gamma^{\alpha, \beta}$ defined by
$[x]_\gamma^{\alpha, \beta}(t)
=\left(t, x(t), \,^CD_{\gamma^1}^{\alpha_1(\cdot,\cdot), 
\beta_1(\cdot,\cdot)}x(t), \ldots, 
\,^CD_{\gamma^n}^{\alpha_n(\cdot,\cdot), 
\beta_n(\cdot,\cdot)}x(t)\right)$.
We assume that the Lagrangian $L:[a,b]\times \mathbb{R}^{n+1} \to \mathbb{R}$ 
is a function of class  $C^{1}$. Along the work, we denote by $\partial_i L$, 
$i \in \{1,\ldots, n+2\}$, the partial derivative of the Lagrangian $L$ 
with respect to its $i$th argument.

Now, we can rewrite functional \eqref{funct_HO_1} as
\begin{equation}
\label{funct_HO_1a}
\mathcal{J}(x,T)=\int_a^T L[x]_\gamma^{\alpha, \beta}(t) dt + \phi(T,x(T)).
\end{equation}
In the sequel, we need the auxiliary notation of the dual fractional derivative:
$$
D_{\overline{\gamma^i}}^{\bi,\ai}=\gamma_2^i \, {_aD_t^{\bi}}
+\gamma_1^i \, {_tD_T^{\ai}},\quad \mbox{where} 
\quad \overline{\gamma^i}=(\gamma_2^i,\gamma_1^i).
$$

In \cite{Tavares_1}, we obtained fractional necessary optimality conditions 
that every local minimizer of functional $\mathcal{J}$, with $n=1$, must fulfill. 
Here, we generalize \cite{Tavares_1} to arbitrary values of $n$, $n \in \mathbb{N}$.

\begin{theorem}[Necessary optimality conditions for \eqref{funct_HO_1}]
\label{HO_teo1}
Suppose that $(x,T)$ gives a minimum to functional \eqref{funct_HO_1a} on $D$.
Then, $(x,T)$ satisfies the following fractional Euler--Lagrange equations:
\begin{equation}
\label{ELeqHO_1}
\partial_2 L[x]_\gamma^{\alpha, \beta}(t)
+\sum_{i=1}^{n} {D_{\overline{\gamma^i}}^{\bi,\ai}}
\partial_{i+2} L[x]_\gamma^{\alpha, \beta}(t)=0,
\end{equation}
on the interval $[a,T]$, and
\begin{equation}
\label{ELeqHO_2}
\sum_{i=1}^{n} \gamma_2^i \left({_aD_t^{\bi}}\partial_{i+2} 
L[x]_\gamma^{\alpha, \beta}(t)-{ _TD{_t^{\bi}}
\partial_{i+2} L[x]_\gamma^{\alpha, \beta}(t)}\right)=0,
\end{equation}
on the interval $[T,b]$. Moreover, $(x,T)$ satisfies the following transversality conditions:
\begin{equation}
\label{CTHO_1}
\begin{cases}
 L[x]_\gamma^{\alpha,\DS \beta}(T)+\partial_1\phi(T,x(T))+\partial_2\phi(T,x(T))x'(T)=0,\\
\DS \sum_{i=1}^{n}\left[\gamma_1^i (-1)^{i-1} \, \frac{d^{i-1}}{dt^{i-1}}  
{_tI_T^{i-\ai}} \partial_{i+2} L[x]_\gamma^{\alpha, \beta}(t)\right.\\
\quad \quad \DS - \left.\gamma_2^i \,\frac{d^{i-1}}{dt^{i-1}}  {_TI_t^{i-\bi}} 
\partial_{i+2} L[x]_\gamma^{\alpha, \beta}(t)\right]_{t=T} +\partial_2 \phi(T,x(T))=0,\\
\DS \sum_{i=j+1}^{n} \left[ \gamma_1^i (-1)^{i-1-j} \, 
\frac{d^{i-1-j}}{dt^{i-1-j}}  {_tI_T^{i-\ai}} \partial_{i+2} L[x]_\gamma^{\alpha, \beta}(t)\right.\\
\quad \quad \DS + \left.\gamma_2^i (-1)^{j+1}\,\frac{d^{i-1-j}}{dt^{i-1-j}}  
{_TI_t^{i-\bi}} \partial_{i+2} L[x]_\gamma^{\alpha, \beta}(t)\right]_{t=T} =0, \quad \forall j=1,\ldots,n-1,\\
\DS \sum_{i=j+1}^{n} \left[\gamma_2^i (-1)^{j+1}\, \left[ 
\frac{d^{i-1-j}}{dt^{i-1-j}} {_aI_t^{i-\bi}}\partial_{i+2} L[x]_\gamma^{\alpha, \beta}(t)\right.\right.\\
\quad \quad \DS \left.\left. - \,\frac{d^{i-1-j}}{dt^{i-1-j}}   {_TI_t^{i-\bi}\partial_{i+2}
L[x]_\gamma^{\alpha, \beta}(t)}\right]\right]_{t=b}=0, \quad \forall j=0,\ldots,n-1.
\end{cases}
\end{equation}
\end{theorem}

\begin{proof}
The proof is an extension of the one found in \cite{Tavares_1}.
Let $h\in C^n([a,b])$ be a perturbing curve and $\triangle T \in\mathbb{R}$ 
an arbitrarily chosen small change in $T$. For a small number $\e \in\mathbb{R} \left(\e\rightarrow 0\right)$, 
if $(x,T)$ is a solution to the problem, we consider an admissible variation of $(x,T)$ of the form $\left(x+\e{h},T+\e\Delta{T}\right)$, and then, by the minimum condition, we have that
$\mathcal{J}(x,T) \leq \mathcal{J}(x+\e{h},T+\e\Delta{T})$.
The constraints $x^{(i)}(a)=x^{(i)}_a$ imply that all admissible variations
must fulfill the conditions $h^{(i)}(a)=0,$ for all $i=0, \ldots, n-1$.
We define function $j(\cdot)$ on a neighborhood of zero by
\begin{equation*}
j(\e) =\mathcal{J}(x+\e h,T+\e \triangle T)
=\int_a^{T+\e \triangle T} L[x+\e h]_\gamma^{\alpha, \beta}(t)\,dt
+\phi\left(T+\e \triangle T,(x+\e h)(T+\e \triangle T)\right).
\end{equation*}
The derivative $j'(\e)$ is
\begin{equation*}
\begin{split}
j'(\e)&=\int_a^{T+\e \triangle T} \left(
\partial_2 L[x+\e h]_\gamma^{\alpha, \beta}(t) h(t)
+\sum_{i=1}^{n}  \partial_{i+2} L[x+\e h]_\gamma^{\alpha, \beta}(t) \DCi h(t) \right)dt\\
&\quad \quad +L[x + {\e h}]_\gamma^{\alpha, \beta}(T + \e \Delta T)\Delta T
+ \partial_1 \phi\left(T+\e \triangle T,
(x+\e h)(T+\e \triangle T)\right) \, \Delta T\\
&\quad\quad +\partial_2\phi\left(T+\e \triangle T,
(x+\e h)(T+\e \triangle T)\right) \, (x+\e h)'(T+\e \triangle T).
\end{split}
\end{equation*}
Hence, a necessary condition for $(x,T)$ 
to be a local minimizer of $j$ is given by $j'(0) =0$, that is,
\begin{multline}
\label{eqHO_derj}
\int_a^T \left( \partial_2 L[x]_\gamma^{\alpha, \beta}(t) h(t)
+\sum_{i=1}^{n} \partial_{i+2} L[x]_\gamma^{\alpha, \beta}(t) \DCi h(t) \right)dt
+L[x]_\gamma^{\alpha, \beta}(T)\Delta T\\
+\partial_1 \phi \left(T, x(T)\right)\Delta T
+ \partial_2\phi(T, x(T))\left[h(t)+x'(T) \triangle T \right]=0.
\end{multline}
Considering the second addend of the integral function \eqref{eqHO_derj}, for $i=1$, we get
\begin{equation*}
\begin{split}
&\int_a^T \partial_3 L[x]_\gamma^{\alpha, \beta}(t) {^CD_{\gamma^1}^{\ao,\bo}} h(t) dt
=\int_a^T \partial_3 L[x]_\gamma^{\alpha, \beta}(t)\left[\gamma_1^1
\, \LCao h(t)+\gamma_2^1 \, \RCbo h(t)\right]dt\\
&=\gamma_1^1 \int_a^T \partial_3 L[x]_\gamma^{\alpha, \beta}(t)\LCao h(t)dt  \\
&\quad + \gamma_2^1 \left[ \int_a^b \partial_3 L[x]_\gamma^{\alpha, \beta}(t) \RCbo h(t)dt
- \int_T^b \partial_3 L[x]_\gamma^{\alpha, \beta}(t) \RCbo h(t)dt \right].
\end{split}
\end{equation*}
Integrating by parts (see Theorem~\ref{thm:FIP_HO}), and since $h(a)=0$, we obtain that
\begin{equation*}
\begin{split}
\gamma_1^1 & \left[ \int_a^T h(t) _tD_T^{\ao} \partial_3
L[x]_\gamma^{\alpha, \beta}(t) dt
+\left[h(t) {_t I_T^{1-\ao}\partial_3 L[x]_\gamma^{\alpha, \beta}(t)}\right]_{t=T} \right]\\
&+  \gamma_2^1 \Biggl[ \int_a^b h(t){_aD_t^{\bo}} \partial_3 L[x]_\gamma^{\alpha, \beta}(t) dt
-\left[h(t){_a I_t^{1-\bo}\partial_3 L[x]_\gamma^{\alpha, \beta}(t)}\right]_{t=b} \\
&\qquad \quad -\Biggl( \int_T^b h(t){_TD_t^{\bo}} \partial_3 L[x]_\gamma^{\alpha, \beta}(t) dt
- \left[h(t){_TI_t^{1-\bo}\partial_3 L[x]_\gamma^{\alpha, \beta}(t)}\right]_{t=b}\\
&\qquad \qquad \quad +\left[h(t){_TI_t^{1-\bo}
\partial_3 L[x]_\gamma^{\alpha, \beta}(t)}\right]_{t=T} \Biggr) \Biggr].
\end{split}
\end{equation*}
Unfolding these integrals, and considering the fractional operator
$D_{\overline{\gamma^1}}^{\beta_1,\alpha_1}$ with
$\overline{\gamma^1}=(\gamma_2^1,\gamma_1^1)$,
then the previous term is equal to
\begin{equation*}
\begin{split}
\int_a^T h(t) & D_{\overline{\gamma^1}}^{\bo,\ao}\partial_3L[x]_\gamma^{\alpha, \beta}(t)dt\\
&+ \int_T^b\gamma_2^1 h(t)\left[_aD_t^{\bo}\partial_3 L[x]_\gamma^{\alpha, \beta}(t)
-{_TD_t^{\bo}\partial_3L[x]_\gamma^{\alpha, \beta}(t)}\right]dt\\
&+h(T) \left[\gamma_1^1 \, {_tI_T^{1-\ao}\partial_3L[x]_\gamma^{\alpha, \beta}(t)}
-{\gamma_2^1 \, {_TI_t^{1-\bo}\partial_3L[x]_\gamma^{\alpha, \beta}(t)}}\right]_{t=T}\\
&- h(b)\gamma_2^1\left[  {_aI_t^{1-\bo}\partial_3 L[x]_\gamma^{\alpha, \beta}(t)}
-{_TI_t^{1-\bo}\partial_3L[x]_\gamma^{\alpha, \beta}(t)}\right]_{t=b}.
\end{split}
\end{equation*}
Now, consider the general case
$\displaystyle \int_a^T \partial_{i+2} 
L[x]_\gamma^{\alpha, \beta}(t) {^CD_{\gamma^i}^{\ai,\bi}} h(t) dt$, $i=3, \ldots, n$. Then, 
\begin{equation*}
\begin{split}
&\gamma_1^i  \left[ \int_a^T h(t)  _tD_T^{\ai} \partial_{i+2}
L[x]_\gamma^{\alpha, \beta}(t) dt +\left. \sum_{k=0}^{i-1} 
(-1)^{k} h^{(i-1-k)}(t) \frac{d^k}{dt^k} {_t I_T^{i-\ai} 
\partial_{i+2} L[x]_\gamma^{\alpha, \beta}(t) t)}\right]_{t=T} \right]\\
&+  \gamma_2^i \Biggl[ \int_a^b h(t){_aD_t^{\bi}} \partial_{i+2} L[x]_\gamma^{\alpha, \beta}(t) dt
+ \left. \left. \sum_{k=0}^{i-1} (-1)^{i+k} h^{(i-1-k)}(t) 
\frac{d^k}{dt^k} {_a I_t^{i-\bi}\partial_{i+2} L[x]_\gamma^{\alpha, \beta}(t)}\right]_{t=b} \right]\\
&-\gamma_2^i \left[ \int_T^b h(t) {_TD_t^{\bi}} \partial_{i+2} L[x]_\gamma^{\alpha, \beta}(t) dt \right.
+ \left. \left. \sum_{k=0}^{i-1} (-1)^{i+k} h^{(i-1-k)}(t) 
\frac{d^k}{dt^k} {_T I_t^{i-\bi}\partial_{i+2} 
L[x]_\gamma^{\alpha, \beta}(t)}\right]^{t=b}_{t=T} \right].
\end{split}
\end{equation*}
Unfolding these integrals, we obtain
\begin{equation*}
\begin{split}
&\int_a^T h(t)  D_{\overline{\gamma^i}}^{\bi,\ai}\partial_{i+2}L[x]_\gamma^{\alpha, \beta}(t)dt\\
&+ \int_T^b\gamma_2^i h(t)\left[_aD_t^{\bi}\partial_{i+2} L[x]_\gamma^{\alpha, \beta}(t)
-{_TD_t^{\bi}\partial_{i+2} L[x]_\gamma^{\alpha, \beta}(t)}\right]dt\\
&+h^{(i-1)}(T)\left[\gamma_1^i \, {_tI_T^{i-\ai}\partial_{i+2}L[x]_\gamma^{\alpha, \beta}(t)}
+{\gamma_2^i (-1)^i \, {_TI_t^{i-\bi}\partial_{i+2}L[x]_\gamma^{\alpha, \beta}(t)}}\right]_{t=T}\\
&+h^{(i-1)}(b) \gamma_2^i (-1)^i \left[ \, {_aI_t^{i-\bi}\partial_{i+2}L[x]_\gamma^{\alpha, \beta}(t)}
- {_TI_t^{i-\bi}\partial_{i+2}L[x]_\gamma^{\alpha, \beta}(t)}\right]_{t=b}\\
&+h^{(i-2)}(T)\left[\gamma_1^i (-1)^1 \frac{d}{dt} \, {_tI_T^{i-\ai}
\partial_{i+2}L[x]_\gamma^{\alpha, \beta}(t)}+{\gamma_2^i (-1)^{i+1} 
\frac{d}{dt} \,{_TI_t^{i-\bi}\partial_{i+2}L[x]_\gamma^{\alpha, \beta}(t)}}\right]_{t=T}\\
&+h^{(i-2)}(b) \gamma_2^i (-1)^{i+1} \left[ \, 
\frac{d}{dt} {_aI_t^{i-\bi}\partial_{i+2}L[x]_\gamma^{\alpha, \beta}(t)}
- \frac{d}{dt} {_TI_t^{i-\bi}\partial_{i+2}L[x]_\gamma^{\alpha, \beta}(t)}\right]_{t=b}\\
&\qquad \vdots\\
&+ h(T)\left[ \gamma_1^i (-1)^{i-1} \frac{d^{i-1}}{dt^{i-1}} {_tI_T^{i-\ai}\partial_{i+2} 
L[x]_\gamma^{\alpha, \beta}(t)}+ \gamma_2^i (-1)^{2i-1}\frac{d^{i-1}}{dt^{i-1}}
{_TI_t^{i-\bi}\partial_{i+2}L[x]_\gamma^{\alpha, \beta}(t)}\right]_{t=T}\\
&+ h(b)\gamma_2^i (-1)^{2i-1}\left[  \frac{d^{i-1}}{dt^{i-1}} {_aI_t^{i-\bi}\partial_{i+2} 
L[x]_\gamma^{\alpha, \beta}(t)}- \frac{d^{i-1}}{dt^{i-1}} {_TI_t^{i-\bi}
\partial_{i+2}L[x]_\gamma^{\alpha, \beta}(t)}\right]_{t=b}.
\end{split}
\end{equation*}
Substituting all the relations into equation \eqref{eqHO_derj}, we obtain that
\begin{equation}
\label{eq_HO_derj2}
\begin{split}
0=&\int_a^T h(t)  \left( \partial_2 L[x]_\gamma^{\alpha, \beta}(t) + \sum_{i=1}^{n}
D_{\overline{\gamma^i}}^{\bi,\ai}\partial_{i+2}L[x]_\gamma^{\alpha, \beta}(t) \right) dt\\
&+ \int_T^b h(t) \sum_{i=1}^{n} \gamma_2^i \left[_aD_t^{\bi}\partial_{i+2} L[x]_\gamma^{\alpha, \beta}(t)
-{_TD_t^{\bi}\partial_{i+2} L[x]_\gamma^{\alpha, \beta}(t)}\right]dt\\
& +\sum_{j=0}^{n-1} h^{(j)}(T) \sum_{i=j+1}^{n}\left[ \gamma_1^i (-1)^{i-1-j} 
\frac{d^{i-1-j}}{dt^{i-1-j}}{_tI_T^{i-\ai}\partial_{i+2} L[x]_\gamma^{\alpha, \beta}(t)} \right.\\
& \quad \quad \quad \quad \quad \quad \quad \quad \left. 
+ \gamma_2^i (-1)^{j+1}\frac{d^{i-1-j}}{dt^{i-1-j}}{_TI_t^{i-\bi}\partial_{i+2}
L[x]_\gamma^{\alpha, \beta}(t)}\right]_{t=T}\\
& +\sum_{j=0}^{n-1} h^{(j)}(b) \sum_{i=j+1}^{n} \gamma_2^i (-1)^{j+1}\left[ 
\frac{d^{i-1-j}}{dt^{i-1-j}} {_aI_t^{i-\bi}\partial_{i+2} 
L[x]_\gamma^{\alpha, \beta}(t)} \right. \\
& \quad \quad \quad \quad \quad \quad \quad \quad \left. 
- \frac{d^{i-1-j}}{dt^{i-1-j}}{_TI_t^{i-\bi}\partial_{i+2}
L[x]_\gamma^{\alpha, \beta}(t)} \right]_{t=b}\\
&+ h(T) \partial_2 \phi \left(T, x(T)\right)
+ \Delta T \left[ L[x]_\gamma^{\alpha, \beta}(T)  +\partial_1 \phi \left(T, x(T)\right)
+ \partial_2\phi(T, x(T))x'(T) \right].
\end{split}
\end{equation}
The fractional Euler--Lagrange equations \eqref{ELeqHO_1}--\eqref{ELeqHO_2} 
and the transversality conditions \eqref{CTHO_1} follow by application of 
the fundamental lemma of the calculus of variations (see, e.g., \cite{Brunt}), 
for appropriate choices of variations.
\end{proof}


\section{Variational problems with time delay}
\label{sec:delay}

In this section, we consider fractional variational problems with time delay. 
As mentioned in \cite{machado_delay}, ``We verify that a fractional derivative 
requires an infinite number of samples capturing, therefore, all the signal history, 
contrary to what happens with integer order derivatives that
are merely local operators. This fact motivates the evaluation of calculation 
strategies based on delayed signal samples''. This subject has already been studied 
for constant fractional order \cite{Almeidadelay,Baleanudelay,Jaraddelay,MR2740195}. 
However, for a variable fractional order, it is, to the authors' best knowledge, an open question. 
We also refer to the works \cite{Daftardar,deng,Laza,wangdelay}, where fractional 
differential equations are considered with a time delay.
For simplicity of presentation, we consider fractional orders 
$\alpha,\beta:[a,b]^2\to(0,1)$. Using similar arguments, 
the problem can be easily generalized for higher-order derivatives. 
Let $\s>0$ and define the vector
$\xs(t)= \left(t, x(t), \DC x(t),x(t-\s)\right)$.
For the domain of the functional, we consider the set
$$
D_\s=\left\{(x,t)\in C^1([a-\s,b])\times [a,b]: \quad \DC x(t) 
\, \mbox{ exists and is continuous on } [a,b]\right\}.
$$
Let $\mathcal{J}:D\to\mathbb R$ be the functional defined by
\begin{equation}
\label{funct:delay}
\mathcal{J}(x,T)=\int_a^T L\xs(t)+\phi(T,x(T))
\end{equation}
subject to the boundary condition $x(t)=\varphi(t)$ 
for all $t\in[a-\s,a]$, where $\varphi$ is a given (fixed) function. 
Again, we assume that the Lagrangian $L$ and the payoff term
$\phi$ are differentiable functions and we denote, for $T\in[a,b]$,
$D_{\overline{\gamma}}^{\b,\a}=\gamma_2$, 
${_aD_t^{\b}}+\gamma_1 \, {_tD_T^{\a}}$,
where $\overline{\gamma}=(\gamma_2,\gamma_1)$.

\begin{theorem}[Necessary optimality conditions for \eqref{funct:delay}]
\label{teo:delay}
Suppose that $(x,T)$ gives a local minimum to functional \eqref{funct:delay} on $D_\s$. 
If $\s\geq T-a$, then $(x,T)$ satisfies
\begin{equation}
\label{ELeq_1delay}
\partial_2 L\xs (t)+D{_{\overline{\gamma}}^{\b,\a}}\partial_3 L\xs (t)=0,
\end{equation}
for $t\in[a,T]$, and
\begin{equation}
\label{ELeq_2delay}
\gamma_2\left({\LDb}\partial_3 L\xs (t)-{ _TD{_t^{\b}}\partial_3 L\xs (t)}\right)=0,
\end{equation}
for $t\in[T,b]$. Moreover, $(x,T)$ satisfies
\begin{equation}
\label{CT1delay}
\begin{cases}
L\xs (T)+\partial_1\phi(T,x(T))+\partial_2\phi(T,x(T))x'(T)=0,\\
\left[\gamma_1 \, {_tI_T^{1-\a}} \partial_3L\xs (t)
-\gamma_2 \, {_TI_t^{1-\b}} \partial_3 L\xs (t)\right]_{t=T}
+\partial_2 \phi(T,x(T))=0,\\
\gamma_2 \left[ {_TI_t^{1-\b}}\partial_3 L\xs (t)
-{_aI_t^{1-\b}\partial_3L\xs (t)}\right]_{t=b}=0.
\end{cases}
\end{equation}
If $\s<T-a$, then Eq. \eqref{ELeq_1delay} is replaced by the two following ones:
\begin{equation}
\label{ELeq_3delay}
\partial_2 L\xs (t)+D{_{\overline{\gamma}}^{\b,\a}}\partial_3 L\xs (t)+\partial_4 L\xs (t+\s)=0,
\end{equation}
for $t\in[a,T-\s]$, and
\begin{equation}
\label{ELeq_4delay}
\partial_2 L\xs (t)+D{_{\overline{\gamma}}^{\b,\a}}\partial_3 L\xs (t)=0,
\end{equation}
for $t\in[T-\s,T]$.
\end{theorem}

\begin{proof} 
Consider variations of the solution $\left(x+\e h,T+\e\Delta{T}\right)$, 
where $h\in C^1([a-\s,b])$ is such that $h(t)=0$ for all $t\in[a-\s,a]$, 
and $\e,\triangle T$ are two reals. If we define
$j(\e)=\mathcal J\left(x+\e h,T+\e\Delta{T}\right)$, then $j'(0)=0$, that is,
\begin{multline}
\label{eq_derjdelay}
\int_a^T \left( \partial_2 L\xs (t) h(t)+ \partial_3 L\xs (t) \DC h(t) 
+ \partial_4 L\xs (t) h(t-\s) \right)dt\\
+L\xs (T)\Delta T+\partial_1 \phi \left(T, x(T)\right)\Delta T
+ \partial_2\phi(T, x(T))\left[h(T)+x'(T) \triangle T \right]=0.
\end{multline}
First, suppose that $\s \geq T-a$. In this case, since
$$
\int_a^T  \partial_4 L\xs (t) h(t-\s) \, dt
=\int_{a-\s}^{T-\s}  \partial_4 L\xs (t+\s) h(t) \, dt
$$
and $h \equiv 0$ on $[a-\s,a]$, this term vanishes in \eqref{eq_derjdelay} 
and we obtain Eq. (5) of \cite{Tavares_1}. The rest of the proof is similar 
to the one presented in \cite[Theorem 3.1]{Tavares_1}, and we obtain 
\eqref{ELeq_1delay}--\eqref{CT1delay}. Suppose now that $\s<T-a$. In this case, 
we have that
\begin{equation*}
\int_a^T  \partial_4 L\xs (t) h(t-\s) \, dt
=\int_{a-\s}^{T-\s}  \partial_4 L\xs (t+\s) h(t) \, dt
=\int_{a}^{T-\s}  \partial_4 L\xs (t+\s) h(t) \, dt.
\end{equation*}
Next, we evaluate the integral
\begin{multline*}
\int_a^T \partial_3 L\xs (t) \DC h(t) \,dt\\
=\int_a^{T-\s} \partial_3 L\xs (t) \DC h(t) \,dt+\int_{T-\s}^T \partial_3 L\xs (t) \DC h(t) \,dt.
\end{multline*}
For the first integral, integrating by parts, we have
\begin{align*}
&\int_a^{T-\s} \partial_3 L\xs (t) \DC h(t) \,dt 
=\gamma_1 \int_a^{T-\s} \partial_3 L\xs (t)\LC h(t)dt\\
&\quad+ \gamma_2 \left[\int_a^b \partial_3 L\xs (t) \RCb h(t)dt
- \int_{T-\s}^b \partial_3 L\xs (t) \RCb h(t)dt \right]\\
&=\int_a^{T-\s} h(t)\left[{\gamma_1} {_tD_{T-\s}^{\a}} 
\partial_3L\xs (t)+{\gamma_2} {_aD_t^{\b}} \partial_3L\xs (t)\right] dt\\
&\quad+\int_{T-\s}^b\gamma_2 h(t)\left[{_aD_t^{\b}} \partial_3 L\xs (t)
-{_{T-\s}D_t^{\b}} \partial_3 L\xs (t)\right] dt\\
&\quad+ \left[h(t)\left[{\gamma_1} {_tI_{T-\s}^{1-\a}} \partial_3L\xs (t)
-{\gamma_2} {_{T-\s}I_t^{1-\b}} \partial_3L\xs (t)\right]\right]_{t=T-\s}\\
&\quad+ \left[\gamma_2 h(t)\left[-{_aI_t^{1-\b}} \partial_3L\xs (t)
+ {_{T-\s}I_t^{1-\b}} \partial_3L\xs (t)\right]\right]_{t=b}.
\end{align*}
For the second integral, in a similar way, we deduce that
\begin{align*}
\int_{T-\s}^T &\partial_3 L\xs (t) \DC h(t) \,dt \\
&=\gamma_1\left[\int_a^T \partial_3 L\xs (t)\LC h(t)dt
- \int_a^{T-\s} \partial_3 L\xs (t)\LC h(t)dt\right]\\
&\quad+ \gamma_2 \left[\int_{T-\s}^b \partial_3 L\xs (t) \RCb h(t)dt
- \int_T^b \partial_3 L\xs (t) \RCb h(t)dt \right]\\
&=\int_a^{T-\s} \gamma_1  h(t)\left[{_tD_T^{\a}} \partial_3L\xs (t)
-{_tD_{T-\s}^{\a}} \partial_3L\xs (t)\right] dt\\
&\quad+\int_{T-\s}^T h(t)\left[{\gamma_1}{_tD_T^{\a}} \partial_3L\xs (t)
+{\gamma_2}{_{T-\s}D_t^{\b}} \partial_3L\xs (t)\right] dt\\
&\quad+\int_T^b\gamma_2 h(t)\left[{_{T-\s}D_t^{\b}} \partial_3 L\xs (t)
-{_TD_t^{\b}} \partial_3 L\xs (t)\right] dt\\
&\quad+ \left[h(t)\left[-{\gamma_1} {_tI_{T-\s}^{1-\a}} \partial_3L\xs (t)
+{\gamma_2} {_{T-\s}I_t^{1-\b}} \partial_3L\xs (t)\right]\right]_{t=T-\s}\\
&\quad+ \left[h(t)\left[{\gamma_1} {_tI_T^{1-\a}} \partial_3L\xs (t)
-{\gamma_2} {_TI_t^{1-\b}} \partial_3L\xs (t)\right]\right]_{t=T}\\
&\quad+ \left[\gamma_2 h(t)\left[-{_{T-\s}I_t^{1-\b}} \partial_3L\xs (t)
+ {_TI_t^{1-\b}} \partial_3L\xs (t)\right]\right]_{t=b}.
\end{align*}
Replacing the above equalities into \eqref{eq_derjdelay}, we prove that
\begin{align*}
0= &\int_a^{T-\s} h(t)\left[\partial_2 L\xs (t)
+ D_{\overline{\gamma}}^{\b,\a}\partial_3 L\xs (t)+\partial_4 L\xs (t+\s)\right]dt\\
&+\int_{T-\s}^Th(t)\left[\partial_2 L\xs (t)+ D_{\overline{\gamma}}^{\b,\a}\partial_3 L\xs (t)\right]dt\\
&+ \int_T^b \gamma_2 h(t) \left[_aD_t^{\b}\partial_3 L\xs (t)-{_TD_t^{\b}\partial_3 L\xs (t)}\right]dt\\
&+ h(T)\left[\gamma_1 \, {_tI_T^{1-\a}\partial_3L\xs (t)}-{\gamma_2 \, {_TI_t^{1-\b}\partial_3L\xs (t)}}
+\partial_2 \phi(t,x(t))\right]_{t=T}\\
&+\Delta T \left[ L\xs (t)+\partial_1\phi(t,x(t))+\partial_2\phi(t,x(t))x'(t)  \right]_{t=T}\\
&+ h(b)\left[ \gamma_2  \left( _TI_t^{1-\b}\partial_3 L\xs (t)-{_aI_t^{1-\b}
\partial_3L\xs (t)}\right) \right]_{t=b}.
\end{align*}
By the arbitrariness of $h$ in $[a,b]$ and of $\triangle T$, 
we obtain Eqs. \eqref{ELeq_2delay}--\eqref{ELeq_4delay}.
\end{proof}


\section{Examples}
\label{sec:examples}

We provide two illustrative examples. Example~\ref{ex:5.1}
is covered by Theorem~\ref{HO_teo1} while Example~\ref{ex:5.2}
illustrates Theorem~\ref{teo:delay}.

\begin{example}
\label{ex:5.1}
Let $p_{n-1}(t)$ be a polynomial of degree $n-1$. If $\alpha,\beta:[0,b]^2\to(n-1,n)$ 
are the fractional orders, then $\DC p_{n-1}(t)=0$ since $p_{n-1}^{(n)}(t)=0$ for all $t$. 
Consider
$$
\mathcal{J}(x,T)=\int_0^T\left[ \left( \DC x(t)\right)^2+(x(t)-p_{n-1}(t))^2-t-1\right]\,dt +T^2
$$
subject to the initial constraints
$x(0)=p_{n-1}(0)$ and $x^{(k)}(0)=p_{n-1}^{(k)}(0)$, $k=1,\ldots,n-1$.
Observe that, for all $t\in[0,b]$,
$\partial_2 L [x]_\gamma^{\alpha, \beta}(t)=2(x(t)-p_{n-1}(t))$ 
and 
$\partial_3 L [x]_\gamma^{\alpha, \beta}(t) = 2 \, \DC x(t)$. 
Therefore, $\partial_i L [p_{n-1}]_\gamma^{\alpha, \beta}(t)=0$, $i=2,3$. 
Also, the first transversality condition reads as 
$$
\left(\DC x(T)\right)^2+(x(T)-p_{n-1}(T))^2-T-1+2T=0,
$$
which is verified at $(x,T)=( p_{n-1},1)$.
Thus, function $x\equiv p_{n-1}$ and the final time $T=1$ satisfy 
the necessary optimality conditions of Theorem~\ref{HO_teo1}. We also 
remark that one has
$$
\mathcal{J}(x,T)\geq\int_0^T\left[-t-1\right]\,dt +T^2=\frac{T^2}{2}-T
$$
for any curve $x$, which attains a minimum value $-1/2$ at $T=1$. 
Since $\mathcal{J}(p_{n-1},1)=-1/2$, we conclude that $(p_{n-1},1)$ 
is the (global) minimizer of $\mathcal{J}$.
\end{example}

\begin{example}
\label{ex:5.2}
Let $\alpha,\beta:[0,b]^2\to(0,1)$, $f$ be a function of class $C^1$, 
and $\hat{f}(t)= \DC f(t)$. Define $\mathcal{J}$ as
\begin{equation*}
\mathcal{J}(x,T)
=\int_0^T\Biggl[ \left( \DC x(t)-\hat{f}(t)\right)^2+(x(t)-f(t))^2
+(x(t-1)-f(t-1))^2-t-2\Biggr]\,dt +T^2
\end{equation*}
subject to the condition $x(t)=f(t)$ for all $t\in[-1,0]$.
In this case, we can easily verify that $(x,T)=(f,2)$ satisfies all the conditions
in Theorem~\ref{teo:delay}, because $\partial_i L_\s[p_{n-1}]_\gamma^{\alpha, \beta} (t)=0$, 
$i=2,3,4$, and that it is actually the (global) minimizer of the problem.
\end{example}


\section{Conclusion}
\label{sec:conc}

We studied two different types of fractional variational problems of variable order: 
of higher-order and with a time delay. Necessary optimality conditions of
Euler--Lagrange type for such functionals, containing a combined 
Caputo derivative of variable fractional order, were established.
To solve such fractional differential equations, we presented a numerical 
method to deal with the fractional operators, based on the open source 
software package \textsf{Chebfun}. Examples were discussed in order to 
illustrate the new findings and the computational aspects of the work.


\appendix

\section{\textsf{Chebfun} code}
\label{Chebfun:code}

\textsf{Chebfun} is an open source software package that ``aims
to provide numerical computing with functions'' in \textsf{MATLAB} \cite{Matlab}. 
For the mathematical underpinnings of \textsf{Chebfun}, we refer the reader
to \cite{chebfun:book}. For the algorithmic backstory of \textsf{Chebfun}, 
we refer to \cite{ChebfunGuide}. In this appendix, we provide 
our own \textsf{Chebfun} code for the variable order fractional calculus.


\subsection{Higher-order Caputo fractional derivatives}
\label{Chebfun:Caputo:der}

The following code implements operator \eqref{eq:left:Cap:der}:
{\small
\begin{verbatim}
  function r = leftCaputo(x,alpha,a,n)
    dx = diff(x,n);
    g = @(t,tau) dx(tau)./(gamma(n-alpha(t,tau)).*(t-tau).^(1+alpha(t,tau)-n));
    r = @(t) sum(chebfun(@(tau) g(t,tau),[a t],'splitting','on'),[a t]);
  end
\end{verbatim}}
\noindent Similarly, we have defined \eqref{eq:right:Cap:der} with \textsf{Chebfun}
in \textsf{MATLAB} as follows:
{\small
\begin{verbatim}
  function r = rightCaputo(x,alpha,b,n)
    dx = diff(x,n);
    g = @(t,tau) dx(tau)./(gamma(n-alpha(tau,t)).*(tau-t).^(1+alpha(tau,t)-n));
    r = @(t) (-1).^n .* sum(chebfun(@(tau) g(t,tau),[t b],'splitting','on'),[t b]);
  end
\end{verbatim}}
\noindent For examples on the use of functions \texttt{leftCaputo} 
and \texttt{rightCaputo}, see Examples~\ref{ex:cheb01} and \ref{ex:cheb01b}.


\subsection{Riemann--Liouville fractional integrals}
\label{Chebfun:RL:int}

Follows our \textsf{Chebfun}/\textsf{MATLAB} code 
corresponding to Definition~\ref{def:RL:fi}.
The \texttt{leftFI.m} file is given by
{\small
\begin{verbatim}
  function r = leftFI(x,alpha,a)
    g = @(t,tau) x(tau)./(gamma(alpha(t,tau)).*(t-tau).^(1-alpha(t,tau)));
    r = @(t) sum(chebfun(@(tau) g(t,tau),[a t],'splitting','on'),[a t]);
  end
\end{verbatim}}
\noindent while the \texttt{rightFI.m} file reads
{\small
\begin{verbatim}
  function r = rightFI(x,alpha,b)
    g = @(t,tau) x(tau)./(gamma(alpha(tau,t)).*(tau-t).^(1-alpha(tau,t)));
    r = @(t) sum(chebfun(@(tau) g(t,tau),[t b],'splitting','on'),[t b]);
  end
\end{verbatim}}
\noindent See Example~\ref{ex:cheb02}. 


\subsection{Higher-order combined fractional Caputo derivative}
\label{Chebfun:combined}

The combined Caputo derivative, as the names indicates, 
combines both left and right Caputo derivatives, that is,
we make use of functions provided in Appendix~\ref{Chebfun:Caputo:der}:
{\small
\begin{verbatim}
  function r = combinedCaputo(x,alpha,beta,gamma1,gamma2,a,b,n)
    lc = leftCaputo(x,alpha,a,n);
    rc = rightCaputo(x,beta,b,n);
    r = @(t) gamma1 .* lc(t) + gamma2 .* rc(t);
  end
\end{verbatim}}
\noindent See Example~\ref{ex:matlab:comb} for the use of function \texttt{combinedCaputo}.


\bigskip

\small


\noindent \textbf{Acknowledgments.}
This work is part of first author's Ph.D., which is carried out at the
University of Aveiro under the Doctoral Programme
\emph{Mathematics and Applications} of Universities of Aveiro and Minho.
It was supported by Portuguese funds through the
\emph{Center for Research and Development in Mathematics and Applications} (CIDMA),
and \emph{The Portuguese Foundation for Science and Technology} (FCT),
within project UID/MAT/04106/2013. Tavares was also supported
by FCT through the Ph.D. fellowship SFRH/BD/42557/2007.
The authors are grateful to two anonymous referees 
for helpful comments and suggestions. 



\end{document}